\documentclass [12pt]{amsart}
\usepackage[utf8]{inputenc}
\usepackage{amsmath,amsthm}
\usepackage{amssymb}
\usepackage{amsthm}
\usepackage{graphicx}
\usepackage{amsfonts, amssymb}
\usepackage[utf8]{inputenc} 
\usepackage[english]{babel}
\usepackage[enableskew]{youngtab}
\newtheorem{theorem}{Theorem}
\newtheorem{lemma}[theorem]{Lemma}

\newtheorem{corollary}[theorem]{Corollary}
\newtheorem{remark}[theorem]{Remark}
\theoremstyle{definition}
\newtheorem{definition}[theorem]{Definition}

\def\<{{\langle}}
\def\>{{\rangle}}

\begin{document}

\title{A Littlewood-Richardson Rule for Dual Stable Grothendieck Polynomials}
\author{Pavel Galashin}

\date{\today}
\address{Department of Mathematics, Massachusetts Institute of Technology, Cambridge, MA, 02139}
\email{galashin@mit.edu}

\maketitle

\begin{abstract}
For a given skew shape, we build a crystal graph on the set of all reverse plane partitions that have this shape. As a consequence, we get a simple extension of the Littlewood-Richardson rule for the expansion of the corresponding dual stable Grothendieck polynomial in terms of Schur polynomials.
\end{abstract}

\maketitle

\def\l{{\lambda}}
\def\m{{\mu}}
\def\lm{{\l/\m}}
\def\R{{\mathcal{R}}}

\section{Introduction}
Dual stable Grothendieck polynomials $g_\lm(x)$ were introduced in \cite{LP2007}. They are Hopf-dual to the stable Grothendieck polynomials, which represent some classes of the structure sheaves of Schubert varieties. The connection of stable and dual stable Grothendieck polynomials with the $K$-theory of the Grassmannian has been discussed in various papers including \cite{LS1982,FK1996,B2002}, and \cite{LP2007}. The paper \cite{LP2007} gives an explicit combinatorial rule for the coefficients of polynomials $g_\l(x)$ in the basis of Schur polynomials $s_\m(x)$. We extend this result to the case of $g_\lm(x)$ for a skew shape $\lm$, and give a different rule (for straight shapes, it coincides with the rule of \cite{B2012}) that provides the same coefficients for straight shapes and extends the classical Littlewood-Richardson rule.  We do this by constructing a crystal graph (see \cite{K1995}) on the set $\R(\lm)$ of all reverse plane partitions of shape $\lm$ with entries not exceeding a fixed number $m>0$.

\subsection{Main results}
To a reverse plane partition $T\in \R(\lm)$ we assign a reading word $r(T)$ in the following way: ignore each entry of $T$ that is equal to the entry directly below it; then read all the remaining entries in the left-to-right bottom-to-top order (the usual reading order for the Young tableaux). After that we define a family of operators $e_1,e_2,\dots,e_{m-1}$ on the set $\R(\lm)$ which are essentially the usual parenthesation operators applied to the reading word (see \cite{LS1978}). 
\begin{theorem}
 \label{thm:crystal}
 The operators $e_1,e_2,\dots, e_{m-1}$ satisfy the crystal axioms (which can be found in \cite{K1995} and will also be discussed in the sequel). 
\end{theorem}

Therefore we get a crystal graph structure on $\R(\lm)$. As a direct application of that (see \cite{K1995}), we get a Littlewood-Richardson rule for the reverse plane partitions:

\begin{corollary}
 \label{cor:LR}
 The dual stable Grothendieck polynomial $g_\lm(x)$ is expanded in terms of Schur polynomials $s_\nu(x)$ as follows:
 $$ g_\lm(x)=\sum_\nu h_\lm^\nu s_\nu(x),$$
 where the sum is over all Young diagrams $\nu$, and the coefficient $h_\lm^\nu$ is equal to the number of reverse plane partitions $T$ of shape $\lm$ and weight $\nu$ such that the reading word $r(T)$ is a lattice word.
\end{corollary}

We also give a self-contained proof of this Corollary without using the theory of crystal graphs. Note that the highest degree homogeneous component of $g_\lm(x)$ is the skew-Schur polynomial $s_\lm(x)$, so Corollary \ref{cor:LR} is an extension of the Littlewood-Richardson rule for skew-Schur polynomials.

\begin{remark}
\def\ceq{{\mathrm{ceq}}}
 In \cite{GGL2015}, the following refinement $\tilde g_\lm(x;t)$ of $g_\lm(x)$ was introduced. For a reverse plane partition $T\in\R(\lm)$ let $\ceq(T):=(c_1,c_2,\dots)$ be a weak composition whose $i$-th entry $c_i$ is equal to the number of columns $j$ such that the boxes $(i,j)$ and $(i+1,j)$ both belong to $\lm$ and the entries of $T$ in these boxes are the same. Let $t=(t_1,t_2,\dots)$ be a vector of indeterminates, and put $t^{\ceq(T)}:=t_1^{c_1}t_2^{c_2}\dots$. Then the bounded degree power series $\tilde g_\lm(x;t)$ is defined as a sum over all reverse plane partitions $T$ of shape $\lm$ of $x^Tt^{\ceq(T)}$. It will be clear later that the operators $e_1,e_2,\dots,e_{m-1}$ preserve this $\ceq$-statistic, therefore, Corollary \ref{cor:LR} also admits a refinement:
 $$\tilde g_\lm(x;t)=\sum_\alpha t^\alpha \sum_\nu h_\lm^{\nu,\alpha} s_\nu(x),$$
 where the first sum is over all weak compositions $\alpha$, and $h_\lm^{\nu,\alpha}$ counts the number of reverse plane partitions $T$ of shape $\lm$ and weight $\nu$ such that the reading word $r(T)$ is a lattice word with an extra property that $\ceq(T)=\alpha$.
\end{remark}

\subsection{Previous research}
There already is a combinatorial rule for the coefficients $h_\lm^\nu$ in \cite{LP2007} for the case when $\m=\emptyset$ and $\lm=\l$ is a straight shape. Namely, $h_\l^\nu$ equals to the number $f_\l^\nu$ of \textit{elegant fillings} of $\l/\nu$, that is, the number of semi-standard Young tableaux $T$ of shape $\l/\nu$ such that all entries in the $i$-th row of $T$ are strictly less than $i$. This formula is Hopf-dual to the corresponding formula for stable Grothendieck polynomials that appeared earlier in \cite[Theorem 2.16]{L2000}, which implies that the dual stable Grothendieck polynomials are indeed Hopf-dual to the usual stable Grothendieck polynomials. To prove this rule, Lam and Pylyavskyy in \cite{LP2007} construct a weight preserving bijection between reverse plane partitions of shape $\l$ and pairs $(S,U)$, where $S$ is a semi-standard Young tableau of some shape $\mu$ and $U$ is an elegant filling of $\l/\mu$. Following this bijection one can deduce that $T$ is a reverse plane partition of shape $\l$ and weight $\nu$ whose reading word is a lattice word if and only if it corresponds to a pair $(S,U)$ such that $S$ is the filling of the shape $\nu$ with all entries in the $i$-th row equal to $i$, and $U$ is an elegant tableau of shape $\l/\nu$. Therefore the bijection from \cite{LP2007} restricted to the reverse plane partitions whose reading word is a lattice word proves the equality of the numbers $h_\l^\nu$ and $f_\l^\nu$. 

For straight shapes, a combinatorial rule that involved the coefficients $h_\l^\nu$ instead of $f_\l^\nu$ was given in \cite[Proposition 5.3]{B2012} together with bijections that also show the equality of the numbers $h_\l^\nu$ and $f_\l^\nu$. 

\subsection{The structure of the paper}
The rest of this section contains some background information about dual stable Grothendieck polynomials, crystal graphs, and introduction to the operators $e_i$ that occur in the statement of Theorem \ref{thm:crystal}. 

The second section is dedicated to the proof of Theorem \ref{thm:crystal} and Corollary \ref{cor:LR} by exploring further properties and connections between the reading words of reverse plane partitions and the action of operators $e_i$.  

\subsection{Preliminaries}
\subsubsection{Reverse plane partitions}

We follow the notations of \cite{LP2007}. Let $\lm$ be a skew shape and $m$ a positive integer. A \textit{reverse plane partition} $T$ of shape $\lm$ with entries in $[m]:=\{1,\dots,m\}$ is a tableau of this shape such that its entries do not exceed $m$ and weakly increase both in rows and in columns. For $i\in [m]$, by $T(i)$ we denote the number of columns of $T$ that contain $i$. To each reverse plane partition $T$ we attach a monomial $x^T=\Pi_{i\in[m]} x_i^{T(i)}$. For a skew shape $\lm$, define the dual stable Grothendieck polynomial $g_\lm(x_1,\dots,x_m)$ as the sum of weights of all reverse plane partitions $T$ of shape $\lm$ with entries in $[m]$:
$$g_\lm(x)=\sum_T x^T.$$
As it was shown in \cite{LP2007}, these polynomials are symmetric.

\subsubsection{Crystal graphs}
\label{subsubsection:crystal}
\def\S{{\mathcal{S}}}
Crystal graphs are important for representation theory of certain quantized universal enveloping algebras, and have been a fruitful topic of research for the past two decades. We give a brief adaptation of the crystal graph theory based on \cite{K1995,S2003,L1994} with a very low yet sufficient for the rest of this paper level of detail. 

A crystal graph $G$ can be viewed as a set $V$ of vertices together with a set $e_1,\dots,e_{m-1}:V\to V\cup \{0\}$ of operators that act on the vertices of $G$ and return either a vertex of $G$ or zero. In addition, these operators are required to satisfy a set of simple \textit{crystal axioms}. If they do, then they are called \textit{crystal operators}, and $G$ is called \textit{a crystal graph}.

Instead of providing the list of these axioms, we give an important example of a crystal graph, which is the only crystal graph that we will be interested in. Fix $n>0$. Let $\S:=[m]^n$ be the set of all strings of length $n$ in the alphabet $[m]$. For $s=(s_1,s_2,\dots,s_n)\in\S$, the weight $w(s)=(w_1(s),\dots,w_m(s))$ is defined as
$$w_i(s):=\# \{j\in[n]:s_j=i\}.$$ 
For $i\in [m-1]$ we define the operator $E_i:\S\to\S\cup\{0\}$. For $s:=(s_1,s_2,\dots, s_n)\in \S$ the value $E_i(s)$ is evaluated using the following algorithm:
\begin{enumerate}
 \item \label{step:ignore}Ignore all entries of $s$ other than the ones equal to $i$ or to $i+1$;
 \item \label{step:pair}Ignore all occurrences of $i+1$ immediately followed by $i$;
 \item \label{step:replace}After doing the previous step as many times as possible we obtain a string that consists of several $i$'s followed by several $i+1$'s. If there is at least one $i+1$, then $E_i$ replaces the leftmost $i+1$ by an $i$, and otherwise we set $E_i(s):=0$.
\end{enumerate}
In other words, $E_i$ labels each $i$ by a closing parenthesis, each $i+1$ by an opening parenthesis, and then it replaces the leftmost unmatched opening parenthesis by a closing one if there are any unmatched opening parentheses present. As an example, let $i=1,m=3,n=13$ and consider the following string $s:=(1,2,2,3,1,3,2,2,2,1,3,1,2)$. After step (\ref{step:ignore}) we ignore all $3$'s, so the string $s$ becomes $(1,2,2,*,1,*,2,2,2,1,*,1,2)$. Here the ignored entries are represented as stars. Next, we do step \ref{step:pair} as many times as needed, so our string is modified as follows:
\begin{eqnarray*}
   s&=&  (1,2,2,3,1,3,2,2,2,1,3,1,2)\\
   &\to& (1,2,2,*,1,*,2,2,2,1,*,1,2)\\
   &\to& (1,2,*,*,*,*,2,2,2,1,*,1,2)\\
   &\to& (1,2,*,*,*,*,2,2,*,*,*,1,2)\\
   &\to& (1,2,*,*,*,*,2,2,*,*,*,1,2)\\
   &\to& (1,2,*,*,*,*,2,*,*,*,*,*,2).
\end{eqnarray*}
Now we can easily calculate the $E_1$-orbit of $s$:
\begin{eqnarray*}
   E_1^0(s)&=&  (1,2,2,3,1,3,2,2,2,1,3,1,2)\\
   E_1^1(s)&=&  (1,\mathbf{1},2,3,1,3,2,2,2,1,3,1,2)\\
   E_1^2(s)&=&  (1,1,2,3,1,3,\mathbf{1},2,2,1,3,1,2)\\
   E_1^3(s)&=&  (1,1,2,3,1,3,1,2,2,1,3,1,\mathbf{1})\\
   E_1^4(s)&=&  0. 
\end{eqnarray*}

\def\Im{{\textrm{Im}}}
\def\Id{{\textrm{Id}}}

Similarly, define the operators $F_i$ to be the operators that replace the rightmost unmatched closing parenthesis by an opening one. The operators $E_i$ and $F_i$ are ``inverse to each other'' in the sense that for any two strings $u,v\in \S$, $E_i(u)=v$ if and only if $F_i(v)=u$. 

These operators satisfy the crystal axioms and therefore have a lot of nice properties, which we summarize in the following Lemma:

\begin{lemma}
\label{lemma:crystal}
\begin{enumerate}
 \item  Each connected component of the corresponding edge-colored graph has exactly one vertex $v\in \S$ such that for every $i\in [m-1]$, $E_i(v)=0$.
 \item This component is completely determined (up to an isomorphism of edge-colored graphs) by the weight $w(v)$, which is clearly a weakly decreasing sequence of integers.
 \item The sum of $x^{w(u)}$ over all vertices $u$ in this connected component is equal to the Schur polynomial $s_{w(v)}$.  
\end{enumerate}
\end{lemma}
Even though all of these properties follow from the fact that $E_i$ and $F_i$ satisfy crystal axioms, we prove them just to make the proof of Corollary \ref{cor:LR} self-contained. Note that a somewhat related proof can be found in \cite{RS1998}.
\begin{proof}
 Note that if the words $u,u'\in \S$ are Knuth equivalent (see \cite{K1970}), then the words $E_i(u)$ and $E_i(u')$ are Knuth equivalent (or both zero), and also the words $F_i(u)$ and $F_i(u')$ are Knuth equivalent (or both zero). And for each word $u\in\S$ there is exactly one word $u'\in\S$ which is Knuth equivalent to $u$ and such that it is a reading word of some semi-standard Young tableau $T$. But the operators $E_i$ and $F_i$ applied to the reading word of $T$ produce a reading word of some other tableau that has the same shape as $T$. 
 
 Now all three properties follow from the fact that any two semi-standard Young tableaux of the same straight shape can be obtained from one another by applying a sequence of operators $E_i$ and $F_i$. To show this, consider a tableau $T_0$ of shape $\l$ such that for every $j$, all of its entries in the $j$-th row are equal to $j$. Consider an integer $k\geq 1$, and let $T$ be a tableau of shape $\l$ such that for $j\geq k$, all entries of $T$ in the $j$-th row are equal to $j$ and such that for $j<k$, the entries of $T$ in the $j$-th row are less than or equal to $k$. Then we claim that such $T$ can be obtained from $T_0$ by applying a sequence of operators $F_i$ for different $i$'s. This statement is true for $k=1$ and can be easily proven by induction for all $k\geq 1$. 
\end{proof}

\subsection{The crystal operators for reverse plane partitions}
\subsubsection{The descent-resolution algorithm}
We describe the descent-resolution algorithm for reverse plane partitions from \cite{GGL2015}, where it was used in order to describe the analogue of the Bender-Knuth involution for reverse plane partitions. Let $\lm$ be a skew shape, and fix $i\in[m-1]$. For a tableau $T'$ of shape $\lm$ such that the entries of $T'$ are equal to either $i$ or $i+1$ and weakly increase in columns but not necessarily in rows, we say that a column of $T'$ is \textit{$i$-pure}, if it contains an $i$ but does not contain an $i+1$. Similarly, we call a column \textit{$i+1$-pure} if it contains an $i+1$ but does not contain an $i$. If a column contains both $i$ and $i+1$, then we call this column \textit{mixed}. 

\begin{definition}[see \cite{GGL2015}]
A tableau $T'$ is a \textit{benign tableau} if the entries of $T'$ weakly increase in columns and for every two mixed columns $A$ and $B$ ($A$ is to the left of $B$), the lowest $i$ in $A$ is not higher than the lowest $i$ in $B$. In other words, the vertical coordinates of the borders between $i$'s and $i+1$'s in mixed columns weakly increase from left to right (see Figure \ref{fig:benign}).
\end{definition}

\newcommand{\mmm}{\multicolumn{1}{|c|}{}}

\newcommand{\mm}[1]{\multicolumn{1}{|c|}{#1}}

\begin{figure}[here]
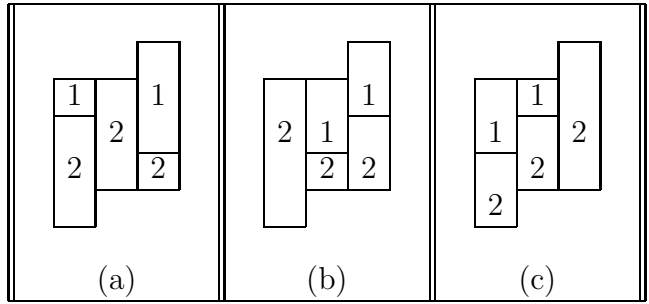

 \centering
\begin{tabular}{||ccc||ccc||ccc||}\hline
 & & & & & & & & \\
 &
\begin{tabular}{ccc}
\cline{3-3}           &               &    \mmm  \\ 
\cline{1-2}  \mm{1}   &   \mmm        &    \mm{1}\\ 
\cline{1-1}   \mmm    &    \mm{2}     &    \mmm  \\ 
\cline{3-3}    \mm{2} &     \mmm      &    \mm{2}\\ 
\cline{2-3}   \mmm    &               &          \\ 
\cline{1-1}\\
\end{tabular}& &
&
\begin{tabular}{ccc}
\cline{3-3}           &               &    \mmm  \\ 
\cline{1-2}  \mmm     &   \mmm        &    \mm{1}\\ 
\cline{3-3}   \mm{2}  &    \mm{1}     &    \mmm  \\ 
\cline{2-2}    \mmm   &     \mm{2}    &    \mm{2}\\ 
\cline{2-3}   \mmm    &               &          \\ 
\cline{1-1}\\
\end{tabular}& &
&
\begin{tabular}{ccc}
\cline{3-3}           &               &    \mmm  \\ 
\cline{1-2}  \mmm     &   \mm{1}      &    \mmm  \\ 
\cline{2-2}   \mm{1}  &    \mmm       &    \mm{2}\\ 
\cline{1-1}    \mmm   &     \mm{2}    &    \mmm  \\ 
\cline{2-3}   \mm{2}  &               &          \\ 
\cline{1-1}   \\
\end{tabular}&\\
& (a) & & & (b) & & & (c) & \\\hline
\end{tabular}

\caption{\label{fig:benign} The table (a) is not benign, (b) is benign but is not a reverse plane partition, (c) is a reverse plane partition.}
\end{figure}

The descent-resolution algorithm takes a benign tableau $T'$ and converts it into a reverse plane partition of the same shape and weight. 

A benign tableau $T'$ may easily fail to be a reverse plane partition. More specifically, it may contain an $i+1$ with an $i$ directly to the right of it -- we call such a situation \textit{a descent}. Let $A$ be the column containing an $i+1$ and $A+1$ be the column containing an $i$. Then there are three possible types of descents depending on the types of the columns $A$ and $A+1$ (their abbreviations are relevant when $i=1$):
\begin{enumerate}
 \item[(2M)] $A$ is $i+1$-pure and $A+1$ is mixed
 \item[(M1)] $A$ is mixed and $A+1$ is $i$-pure
 \item[(21)] $A$ is $i+1$-pure and $A+1$ is $i$-pure
\end{enumerate}
There is a fourth type of descents in which both columns are mixed, but the benign tableau property implies that such descents are impossible. For a descent of each of these three types, \cite{GGL2015} provides a \textit{descent-resolution step}, which changes only the entries of $A$ and $A+1$ and resolves this descent. 

For descents of the first two types, the descent-resolution step switches the roles of the columns but preserves the vertical coordinate of the lowest $i$ in the mixed column; this determines the operation uniquely. For a descent of the third type, it simply replaces all $i$'s by $i+1$'s and vice versa in both columns. It is clear that the resulting tableau will also be a benign tableau. The descent-resolution steps for $i=1$ are visualized in Figure \ref{fig:reduction}.

\def\one{{\mathbf{1}}}
\def\two{{\mathbf{2}}}


\begin{figure}[here]
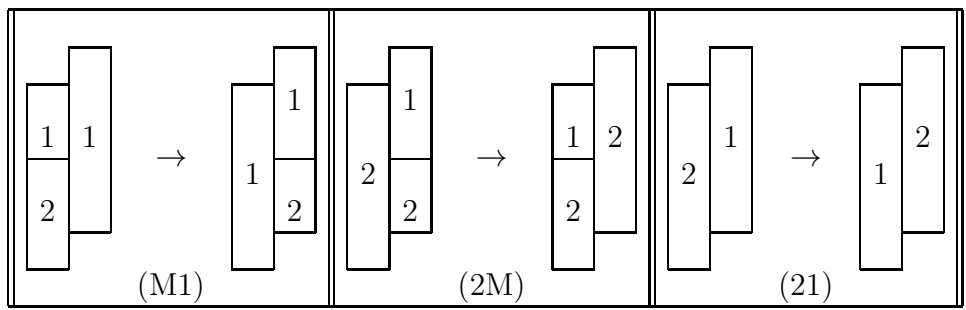

 \centering
\begin{tabular}{||ccc||ccc||ccc||}\hline
 & & & & & & & & \\
\begin{tabular}{cc}
\cline{2-2} & \multicolumn{1}{|c|}{}\\
\cline{1-1} \multicolumn{1}{|c|}{} & \multicolumn{1}{|c|}{} \\
\multicolumn{1}{|c|}{1} & \multicolumn{1}{|c|}{1}\\
\cline{1-1} \multicolumn{1}{|c|}{} & \multicolumn{1}{|c|}{}\\
\multicolumn{1}{|c|}{2} & \multicolumn{1}{|c|}{}\\
\cline{2-2} \multicolumn{1}{|c|}{} \\
\cline{1-1}
\end{tabular}
&
$\rightarrow$
&
\begin{tabular}{cc}
\cline{2-2} & \multicolumn{1}{|c|}{}\\
\cline{1-1} \multicolumn{1}{|c|}{} & \multicolumn{1}{|c|}{1} \\
\multicolumn{1}{|c|}{} & \multicolumn{1}{|c|}{}\\
\cline{2-2} \multicolumn{1}{|c|}{1} & \multicolumn{1}{|c|}{}\\
\multicolumn{1}{|c|}{} & \multicolumn{1}{|c|}{2}\\
\cline{2-2} \multicolumn{1}{|c|}{} \\
\cline{1-1}
\end{tabular}
&
\begin{tabular}{cc}
\cline{2-2} & \multicolumn{1}{|c|}{}\\
\cline{1-1} \multicolumn{1}{|c|}{} & \multicolumn{1}{|c|}{1} \\
\multicolumn{1}{|c|}{} & \multicolumn{1}{|c|}{}\\
\cline{2-2} \multicolumn{1}{|c|}{2} & \multicolumn{1}{|c|}{}\\
\multicolumn{1}{|c|}{} & \multicolumn{1}{|c|}{2}\\
\cline{2-2} \multicolumn{1}{|c|}{} \\
\cline{1-1}
\end{tabular}
&
$\rightarrow$
&
\begin{tabular}{cc}
\cline{2-2} & \multicolumn{1}{|c|}{}\\
\cline{1-1} \multicolumn{1}{|c|}{} & \multicolumn{1}{|c|}{} \\
\multicolumn{1}{|c|}{1} & \multicolumn{1}{|c|}{2}\\
\cline{1-1} \multicolumn{1}{|c|}{} & \multicolumn{1}{|c|}{}\\
\multicolumn{1}{|c|}{2} & \multicolumn{1}{|c|}{}\\
\cline{2-2} \multicolumn{1}{|c|}{} \\
\cline{1-1}
\end{tabular}
&
\begin{tabular}{cc}
\cline{2-2} & \multicolumn{1}{|c|}{}\\
\cline{1-1} \multicolumn{1}{|c|}{} & \multicolumn{1}{|c|}{} \\
\multicolumn{1}{|c|}{} & \multicolumn{1}{|c|}{1}\\
\multicolumn{1}{|c|}{2} & \multicolumn{1}{|c|}{}\\
\multicolumn{1}{|c|}{} & \multicolumn{1}{|c|}{}\\
\cline{2-2} \multicolumn{1}{|c|}{} \\
\cline{1-1}
\end{tabular}
&
$\rightarrow$
&
\begin{tabular}{cc}
\cline{2-2} & \multicolumn{1}{|c|}{}\\
\cline{1-1} \multicolumn{1}{|c|}{} & \multicolumn{1}{|c|}{} \\
\multicolumn{1}{|c|}{} & \multicolumn{1}{|c|}{2}\\
\multicolumn{1}{|c|}{1} & \multicolumn{1}{|c|}{}\\
\multicolumn{1}{|c|}{} & \multicolumn{1}{|c|}{}\\
\cline{2-2} \multicolumn{1}{|c|}{} \\
\cline{1-1}
\end{tabular}\\
& (M1) & & & (2M) & & & (21) & \\\hline
\end{tabular}

\caption{\label{fig:reduction} The descent-resolution steps (taken from \cite{GGL2015}).}
\end{figure}

The descent-resolution algorithm performs these descent-resolution steps until there are no descents left, which means that we get a reverse plane partition. This algorithm terminates, because $i$-pure columns always move to the right while $i+1$-pure columns always move to the left. Also, it is shown in \cite{GGL2015} that the result of the algorithm does not depend on the order in which the descents are resolved.

\subsubsection{The definition of $e_i$'s and $f_i$'s}
Let $\lm$ be a skew shape, and fix $i\in[m-1]$. For a reverse plane partition $T$ of shape $\lm$ with entries in $[m]$, define $e_i(T)$ as follows. First, consider only the subtableau of $T$ that consists of entries equal to either $i$ or $i+1$.  Then, label each $i$-pure column by a closing parenthesis and each $i+1$-pure column by an opening parenthesis (and ignore the mixed columns). 

Choose the ($i+1$-pure) column $A$ that corresponds to the leftmost unmatched opening parenthesis (if all opening parentheses are matched, set $e_i(T):=0$). Replace all the $i+1$'s in $A$ by $i$'s, and then apply the descent-resolution algorithm to the resulting benign tableau. 

Similarly, $f_i$ chooses the ($i$-pure) column $B$ that corresponds to the rightmost unmatched closing parenthesis and replaces all the $i$'s in it by $i+1$'s and then applies the descent-resolution algorithm.

We discuss the properties of $e_i$'s and $f_i$'s and their connection to the defined above reading word in the next section.

\section{Properties of the reading words of reverse plane partitions}
Recall that the reading word $r(T)$ of a reverse plane partition $T$ of shape $\lm$ is defined as the usual left-to-right bottom-to-top Young tableaux reading word that ignores each entry that has the same entry below it. An example is shown in Figure \ref{fig:rw}.

\begin{figure}[h]
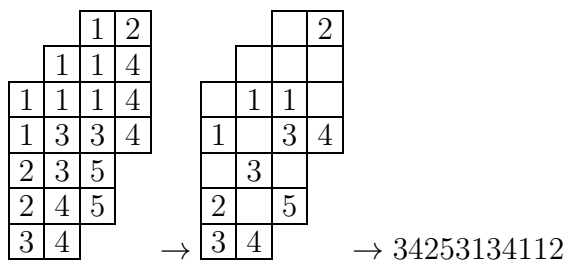

$$\young(::12,:114,1114,1334,235:,245,34)\to \young(::\ 2,:\ \ \ ,\ 11\ ,1\ 34,\ 3\ :,2\ 5,34)\to 34253134112$$
\caption{\label{fig:rw} The reading word of a skew-shaped reverse plane partition.}
\end{figure}

We assume the coordinates of the boxes are in the matrix notation. For a reverse plane partition $T$, define its \textit{height vector} $h(T)$ to be the sequence of vertical coordinates of the entries of $T$ that contribute to $r(T)$ arranged in the exact same order as they appear in the reading word. For example, for $T$ as in Figure \ref{fig:rw} we put $h(T):=(7,7,6,6,5,4,4,4,3,3,1)$. It is always a weakly decreasing sequence of positive integers. Similarly, we define the height vector of a benign tableau. Note that each descent-resolution step preserves the height vector, and, therefore, so do the operators $e_i$ and $f_i$.

\begin{lemma}
 \label{lemma:injective}
 Fix a skew shape $\lm$ and a sequence $h$ of positive integers. Then for each reading word $r$ there is at most one reverse plane partition $T$ of shape $\lm$ with $r(T)=r$ and $h(T)=h$.
\end{lemma}
\begin{proof}
 Suppose that there exists a reverse plane partition $T$ of shape $\lm$ with $r(T)=r$ and $h(T)=h$. Then $T$ can be uniquely reconstructed from $r$ and $h$ by filling the boxes of $\lm$ in the reading order:
 \begin{enumerate}
  \item Set $j=1$;
  \item Let $B$ be the first (in the reading order) box of $\lm$ which is not filled with a number. Let $a$ be the value in the box directly below it, and let $c$ be the value in the box directly to the left of it (if there is no such box then we put $a:=+\infty$ or $c:=0$);
  \item If the height of $B$ is not equal to $h_j$, then set the entry in the box $B$ equal to $a$ and proceed to the next box (go to step 2);
  \item If the number $r_j$ does not satisfy $c\leq r_j<a$, then, again, set the entry in the box $B$ equal to $a$ and proceed to the next box;
  \item Otherwise, we set the entry in the box $B$ equal to $r_j$, increase $j$ by $1$ and proceed to the next box.
 \end{enumerate}
 Note that if $r$ and $h$ are the reading word and the height vector of some reverse plane partition, then the entries of $h$ weakly decrease, and the entries of $r$ that have the same height weakly increase. We prove by induction that the first $k$ entries of $T$ (in the reading order) are the same as the first $k$ entries of the reverse plane partition that the algorithm produces. For $k=0$ it is true. Now, we want to put $r_j$ somewhere into the row $h_j$ so that the entry below it is strictly bigger than $r_j$ and so that the entries in the row weakly increase. Thus if $r_j$ cannot be put into the current box (because either $r_j\geq a$ or $r_j<b$), then this box should be ignored by the reading word, so its value should be the same as the value directly below it. If $b\leq r_j<a$, then $r_j$ has to be put into the current box, because if we put $r_j$ somewhere to the right, then we have to fill this box with the value directly below it (with $a$), which is strictly bigger than $r_j$, so the entries in the row will not be weakly increasing.
\end{proof}

Recall that the operators $F_i$ are defined on the set $\S=[m]^n$ of all strings of length $n$, and replace the rightmost unmatched closing parenthesis (corresponding to an entry equal to $i$) by an opening parenthesis (by an $i+1$). Meanwhile, the operators $f_i$ act on $\R(\lm)$, which is the set of all reverse plane partitions of shape $\lm$ with entries less than or equal to $m$. It turns out that these two actions commute with the operation of taking the reading word:

\begin{lemma}
 \label{lemma:intertw}
 Let $T$ be a reverse plane partition. Then 
 $$F_i(r(T))=r(f_i(T)).$$
 In particular, if $f_i(T)$ is zero then $F_i(r(T))$ is zero and the converse is also true.
\end{lemma}
And, because $e_i$ and $f_i$ are ``inverse to each other'' (in the same sense as above), and the same is true for $E_i$ and $F_i$, we get
\begin{corollary}
 Let $T$ be a reverse plane partition. Then 
 $$E_i(r(T))=r(e_i(T)).$$
\end{corollary}

\begin{proof}[{Proof of Lemma \ref{lemma:intertw}}]
 The operator $f_i$ labels $i$-pure columns by closing parentheses and $i+1$-pure columns by opening parentheses. Then it finds the rightmost unmatched closing parenthesis and replaces the corresponding $i$-pure column by an $i+1$-pure column. After that we get a benign tableau $T'$, and then we apply the descent-resolution algorithm to $T'$ which produces a reverse plane partition $T''=:f_i(T)$. Our proof consists of two parts: 
 \begin{enumerate}
  \item $r(T')=r(T'')$;
  \item $F_i(r(T))=r(T')$.
 \end{enumerate}
\begin{remark}
 Note that both of these parts are false for $e_i$ and $E_i$. To make them true, one needs to introduce the reading word that ignores each entry equal to the entry directly \textit{above} it, rather than directly below it. 
\end{remark}
 We start with the first part. Note that even though $T'$ and $T''$ differ by a sequence of descent-resolution steps, it is not true in general that the descent-resolution steps preserve the reading word. Fortunately, as we will see later, all the appearing descents are of the first type. And the corresponding descent-resolution step (see Figure \ref{fig:reduction1}) clearly does not change the reading word.

\begin{figure}[here]
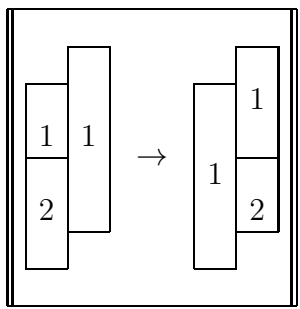

 \centering
\begin{tabular}{||ccc||}\hline
 & &  \\
\begin{tabular}{cc}
\cline{2-2} & \multicolumn{1}{|c|}{}\\
\cline{1-1} \multicolumn{1}{|c|}{} & \multicolumn{1}{|c|}{} \\
\multicolumn{1}{|c|}{1} & \multicolumn{1}{|c|}{1}\\
\cline{1-1} \multicolumn{1}{|c|}{} & \multicolumn{1}{|c|}{}\\
\multicolumn{1}{|c|}{2} & \multicolumn{1}{|c|}{}\\
\cline{2-2} \multicolumn{1}{|c|}{} \\
\cline{1-1}
\end{tabular}
&
$\rightarrow$
&
\begin{tabular}{cc}
\cline{2-2} & \multicolumn{1}{|c|}{}\\
\cline{1-1} \multicolumn{1}{|c|}{} & \multicolumn{1}{|c|}{1} \\
\multicolumn{1}{|c|}{} & \multicolumn{1}{|c|}{}\\
\cline{2-2} \multicolumn{1}{|c|}{1} & \multicolumn{1}{|c|}{}\\
\multicolumn{1}{|c|}{} & \multicolumn{1}{|c|}{2}\\
\cline{2-2} \multicolumn{1}{|c|}{} \\
\cline{1-1}
\end{tabular}\\ & & \\\hline
\end{tabular}
\caption{\label{fig:reduction1} The first descent-resolution step (M1).}
\end{figure}


 The reason we only need this descent-resolution step is the definition of $f_i$. Namely, $f_i$ changes only one $i$-pure column $A$ into an $i+1$-pure column. And this column is required to be labeled by the rightmost unmatched closing parenthesis. Let $B$ be the leftmost $i+1$-pure column to the right of $A$, and let $C$ be the rightmost $i$-pure column to the left of $A$. If there was an $i$-pure column between $A$ and $B$, then it would also be unmatched, so $A$ would not be labeled by the rightmost unmatched closing parenthesis. Also, if there was an $i+1$-pure column $D$ between $C$ and $A$, then it would have to be matched to some $i$-pure column between $D$ and $A$, so $C$ would not be the rightmost $i$-pure column to the left of $A$. All in all we can see that all the columns between $C$ and $A$ and between $A$ and $B$ are mixed. If either $C$ or $B$ is undefined, then all the columns to the left (resp., to the right) of $A$ are mixed.
 
 Now it is clear why only the descents of the first type appear while the descent-resolution steps are performed. The column $A$ becomes $i+1$-pure, so the only possible descent can occur between $A$ and $A+1$, and as we resolve it, the $i+1$-pure column moves to the right. But because it is surrounded by mixed columns, the only appearing descents are between this $i+1$-pure column and the mixed column to the right of it. And if this $i+1$-pure column moves to the position $B-1$, then there are no descents left, because $B$ is also $i+1$-pure. This finishes the proof of the first part.
 
 The second part asks for a certain correspondence between two different matchings. The first one appears when we label $i$-pure columns by closing parentheses, $i+1$-pure columns by opening parentheses, and then say that two pure columns match each other if their labels (two parentheses) match each other in the parenthesis sequence. In this situation we say that these two columns \textit{match in the reverse plane partition}. The second matching appears when we label the entries of the reading word by parentheses and say that two entries of the reading word match each other if their labels match each other. In this situation we say that these two entries \textit{match in the reading word}. 
 
 The second part of the Lemma states that an $i$-pure column is labeled by the rightmost unmatched closing parenthesis in the reverse plane partition 
 if and only if the corresponding entry in the reading word is also labeled by the rightmost unmatched closing parenthesis in the reading word. Here we can restrict our attention to reverse plane partitions that are filled only with $i$'s and $i+1$'s. For a column $A$, let $j(A)$ be the position of the corresponding entry of the reading word if $A$ is either $i$- or $i+1$-pure. If $A$ is mixed, then set $j^-(A)$ to be the position of the entry of the reading word corresponding to $i$ and set $j^+(A)$ to be the position of the entry of the reading word corresponding to $i+1$. 
 
 We need to check three implications:
 \begin{enumerate}
  \item If a column $A$ is $i$-pure and unmatched in the reverse plane partition, then the entry $j(A)$ is unmatched in the reading word.
  \item If a column $A$ is mixed, then the entry $j^-(A)$ is matched to something (not necessarily to $j^+(A)$) in the reading word.
  \item If a column $A$ is $i$-pure and matched to some $i+1$-pure column $B$ in the reverse plane partition, then the entry $j(A)$ is also matched to something (not necessarily to $j(B)$) in the reading word.
 \end{enumerate}
 
 It is clear that these three properties together imply that the $i$-pure columns unmatched in the reverse plane partition correspond exactly to the unmatched $i$'s in the reading word. And because the reading word preserves the order of pure columns, the second part of the lemma reduces to proving these three implications.
 
 Note that if a column $A$ is $i$-pure, then for every other column $B$ that is to the right (resp., left) of $A$, the entry $j(B)$ or $j^-(B)$ or $j^+(B)$ if defined is also to the right (resp. left) of $j(A)$. Another simple useful observation is that if we have any injective map that attaches to each $i+1$ in the reading word an $i$ to the right of it, then all the $i+1$'s in this reading word are matched. Now we are ready to check the implications (1)-(3).
 
 (1) If a column $A$ is $i$-pure and unmatched, then we can just throw everything to the right of $A$ and $j(A)$ out. Now, every $i+1$-pure column to the left of $A$ is matched to something in the reverse plane partition, so for every $i+1$ to the left of $j(A)$ in the reading word we have an $i$ that is between it and $j(A)$, and for different $i+1$'s these $i$'s are also different. Therefore every $i+1$ to the left of $j(A)$ is matched in the reading word as well, so $j(A)$ is unmatched in the reading word.
 
 (2) Suppose $A$ is mixed. If we throw out all the columns that are to the right of $A$, then several $i+1$'s between $j^+(A)$ and $j^-(A)$ will be thrown out of the reading word, but all the $i$'s to the left of $j^-(A)$ will remain untouched. Let $B$ be the rightmost $i$-pure column to the left of $A$. Now we throw out all the columns to the left of $B$ and also $B$ itself, which corresponds to erasing the part the reading word from the beginning to $j(B)$ (if there was no such $B$ then we do not throw anything out of the reading word). Now we have a reverse plane partition that contains no $i$-pure columns, so by the counting argument $j^-(A)$ is matched in the reading word. But then it was also matched in the original reading word.
 
 (3) Suppose $A$ is $i$-pure and is matched in the reverse plane partition to some $i+1$-pure column $B$ to the left of $A$. Let $C$ be the rightmost $i$-pure column to the left of $B$. We throw out everything that is to the right of $A$ or to the left of $C$, which corresponds to keeping all the entries of the reading word between $j(C)$ and $j(A)$. We also remove $C$ and $j(C)$. All the $i$-pure columns between $A$ and $B$ are matched in the reverse plane partition to some $i+1$-pure columns between $A$ and $B$, and there are no $i$-pure columns between $B$ and $C$, so the number of $i+1$'s between $j(C)$ and $j(A)$ is strictly bigger than the number of $i$'s between $j(C)$ and $j(A)$, so $j(A)$ has to be matched to something in the reading word. We finish the proof of the third implication, which finishes the proof of the second (last) part of the Lemma.
\end{proof}

Let $\lm$ be a skew shape, and let $h$ be a sequence of positive integers. Lemmas \ref{lemma:injective} and \ref{lemma:intertw} give a vertex-injective map from the graph of all reverse plane partitions $T$ of shape $\lm$ with $h(T)=h$ to the graph $\S$ of all strings of the same length as $h$, and this map takes the operators $e_i$ and $f_i$ to $E_i$ and $F_i$. Therefore each connected component of the graph of all reverse plane partitions is isomorphic to the corresponding connected component of the graph $\S$. Now the proof of Theorem \ref{thm:crystal} follows from the observations about crystal graphs made in Subsection \ref{subsubsection:crystal}, in particular, the proof of Corollary \ref{cor:LR} follows from Lemma \ref{lemma:crystal}. \qed

\section*{Acknowledgments}
I am grateful to Prof. Alex Postnikov and to Darij Grinberg for their valuable remarks.

\bibliography{rpp}
\bibliographystyle{plain}

\end{document}